\documentclass[a4paper,11pt]{article}
\usepackage{geometry}
\usepackage{amssymb}
\usepackage{amsthm}
\usepackage{amsmath}
\usepackage{hyperref}
\usepackage{authblk}
\usepackage{mathtools}
\usepackage{cite}
\newtheorem{theorem}{Theorem}[section]
\newtheorem{lemma}[theorem]{Lemma}
\newtheorem{open_problem}{Open problem}
\newtheorem{corollary}[theorem]{Corollary}
\theoremstyle{definition}

\newcommand{\Z}{\mathbb{Z}}
\newcommand{\Q}{\mathbb{Q}}
\newcommand{\R}{\mathbb{R}}

\newcommand{\fact}{\mathrm{Fact}}
\newcommand{\inj}{\mathcal{I}}
\newcommand{\ie}{\mathrm{IE}}
\newcommand{\fe}{\mathrm{E}}
\newcommand{\ce}{\mathrm{CE}}
\newcommand{\ace}{\mathrm{ACE}}
\newcommand{\iei}{\ie_\inj}
\newcommand{\fei}{\fe_\inj}
\newcommand{\cei}{\ce_\inj}
\newcommand{\acei}{\ace_\inj}
\newcommand{\aceii}{\ace'_\inj}
\newcommand{\sep}{|}
\newcommand{\eps}{\varepsilon}
\DeclareMathOperator{\alphabet}{alph}

\newcommand{\N}{\mathbb{N}}
\newcommand{\abs}[1]{\left\vert{#1}\right\vert}
\newcommand*\samethanks[1][\value{footnote}]{\footnotemark[#1]}

\title{Mapped Exponent and Asymptotic Critical Exponent of Words}
\author[1]{Eva Foster\thanks{Work completed whilst affiliated with School of Computer Science and Mathematics, Keele University, United Kingdom}}
\author[2]{Aleksi Saarela\thanks{Supported by the Research
Council of Finland under grants 339311, 346566 and 359921}}
\author[2]{Aleksi Vanhatalo\samethanks }
\affil[1]{School of Computing and Mathematical Sciences, Birkbeck College, University of London, United Kingdom}
\affil[2]{Department of Mathematics and Statistics, University of Turku, Finland}

\begin{document}

\maketitle
\begin{abstract}
    We study how much injective morphisms can increase the repetitiveness of a given word. This question has a few possible variations depending on the meaning of ``repetitiveness''. We concentrate on fractional exponents of finite words and asymptotic critical exponents of infinite words. We characterize finite words that, when mapped by injective morphisms, can have arbitrarily high fractional exponent. For infinite words, alongside other results, we show that the asymptotic critical exponent grows at most by a constant factor (depending on the size of the alphabet) when mapped by an injective morphism. For both finite and infinite words, the binary case is better understood than the general case.
\end{abstract}

\section{Introduction}

The study of different variations of repetitions and the study of morphisms
are two of the most common and old areas in combinatorics on words,
dating back to Thue's original papers in 1906 \cite{thue1906uber}.
Since then, properties of repetitions and morphisms have been studied extensively,
both as separate areas and in relation to each other.
A common viewpoint is to follow Thue's footsteps
and use morphisms to generate words with desired properties,
one of which is often a variation of avoiding repetitions,
see \cite{karhumaki1983cube,currie2008each} for some examples.

A variation which is relevant for this article
is to consider infinite words where small (in length) exceptions to repetition-freeness are allowed.
For example, there exist infinite binary words that have a small number of short square factors
but are otherwise square-free \cite{entringer1974nonrepetitive}.
Many articles study similar questions, e.g., \cite{harju2006binary}.
Taking this idea further leads to the concept of asymptotic critical exponent (ACE).
It is known that there are binary infinite words with ACE equal to one
\cite{beck1984,cassaigne2008extremal}.
The asymptotic critical exponent has also been considered (sometimes under different names)
in other papers, such as \cite{vandeth2000sturmian,dolce2021balanced}.

In this paper, instead of using morphisms to generate words with few repetitions,
we study questions of the following type:
How good are morphisms at introducing some form of repetition into a given word?
How good are words at resisting being transformed in this way by morphisms?
If the morphisms in question are not restricted in any way, the answer is trivial,
as we can just map all words to a unary alphabet.
However, maybe a bit surprisingly at the first glance,
restricting our attention to injective morphisms leads to many interesting results.

We consider two variations of this topic: one about finite and the other about infinite words.
The research on finite words was started by the second author in \cite{Saarelapowers}.
Explicitly he asked, given a word,
what natural numbers can appear as exponents of that word
after it has been mapped by arbitrary, nonperiodic or injective morphisms.
In this article, we study a fairly direct extension of the problem,
where instead of integer exponents, we consider all rational exponents.
The main question we then try to answer is the following:
For which words $w$ is the fractional exponent of $h(w)$ bounded by a finite number
when $h$ runs through all injective morphisms?
We also introduce some extremal example words for which the optimal value of such a bound
is either large or arbitrarily close to one.

The second variation is an attempt to formalize a similar problem for infinite words.
Morphisms, even injective ones, can trivially introduce large repetitions as factors,
so studying ordinary repetition-freeness is not interesting here.
The concept of asymptotic critical exponent, however, is well-suited for study.
Among other things, we are able to prove an upper bound
for how large the ratio of the ACE of $h(w)$ and the ACE of $w$ can be
when $h$ is an injective morphism.
We also give examples of words where this ratio is large,
and examples of words such that the ACE of $h(w)$ is exactly one for all injective morphisms $h$.
In the case of a binary alphabet, we obtain particularly strong results.

\subsection{Preliminaries}

While we define most of the terminology we use and the paper is mostly self-contained, some of our proofs can be tricky for readers unexperienced in combinatorics on words.
For a general reference, see \cite{Lothaire}. Some definitions are found outside of this section just before we use them for the first time.

\emph{Factor} $f$ of a (finite or infinite) word $w$ is a word such that $w=ufv$ for some words $u$ and $v$. We denote the set of length-$n$ factors of a word $w$ by $\fact_n(w)$,
and the set of all nonempty factors by $\fact_+(w)$. We also use $\alphabet(w)$ to denote the set of letters that occur in $w$. If $f$ is a factor of $w$ such that $w=fv$ for some $v$, then $f$ is called a \emph{prefix} of the word $w$. Similarly, if $w=uf$ for some word $u$ then $f$ is a \emph{suffix} of the word $w$.

Any equality where a word $w \in \Sigma^*$ is written in terms of two or more shorter words, e.g. $w=w_1w_2 \dotsm w_k$ for some words $w_i \in \Sigma^*$ is called a \emph{factorization} of $w$. Often when talking about factorization of some fixed word, we mean some particular factorization that is clear from the context.

Let $v \in \Sigma^+$ be a word and $p,q\in \N$ with $q=\abs{v}$. By $v^\omega$ we denote the infinite word $vvv \dotsm$ and the word $v^{\frac{p}{q}}$ then denotes the prefix of length $p$ of $v^\omega$. With this notation we have, for example, $(ab)^3=(ab)^{\frac{6}{2}}=(abab)^{\frac{6}{4}}=(abab)^{\frac{3}{2}}=ababab$ and $(abc)^{\frac{4}{3}}=abca$.
For a nonempty word $u \in \Sigma^+$,
its \emph{integer exponent} $\ie(u)$
and \emph{fractional exponent} $\fe(u)$
are then defined by:
\begin{align*}
    \ie(u) &= \sup \{n \in \Z \mid \exists v \in \Sigma^+: u = v^n\},\\
    \fe(u) &= \sup \{r \in \Q \mid \exists v \in \Sigma^+: u = v^r\}.
\end{align*}
We may also use the term \emph{exponent} for $\fe(u)$
(in the literature, exponent is sometimes used for integer exponent and sometimes for fractional exponent). If $\ie(x)=1$, then we say that $x$ is \emph{primitive}. Note that if $w=x^{\fe(w)}$, then $x$ is primitive since $(y^n)^r=y^{rn}$ if $n\in \N$. If $w=x^{\ie(w)}$, then we say that $x$ is a \emph{primitive root} of $w$. We say that non-empty words $u,v$ are \emph{conjugate}, if $u=pq$ and $v=qp$ for some words $p,q$. Equivalently words $u$ and $v$ are conjugate if $\abs{u}=\abs{v}$ and $v$ is a factor of $uu$.
Following three well-known facts can be found in the first few pages of \cite{Lothaire}; \newline
1. Two words $u$ and $v$ commute, i.e. satisfy $uv=vu$, only if words $u$ and $v$ are integer powers of a common word, \newline
2. If words $x$ and $y$ are conjugate, then their primitive roots are conjugate and \newline
3. Conjugates of a primitive word are primitive. (This follows from the fact 2.)

We use these facts without explicitly stating so, for example by declaring a contradiction if we show that a conjugate of a primitive word is nonprimitive.

For an infinite word $w$,
its \emph{critical exponent} $\ce(w)$
and \emph{asymptotic critical exponent} $\ace(w)$
are defined by:
\begin{align*}
    \ce(w) &= \sup \{\fe(u) \mid u \in \fact_+(w)\},\\
    \ace(w) &= \limsup_{n \to \infty} \{\fe(u) \mid u \in \fact_n(w)\}.
\end{align*}

Let us fix two finite alphabets $\Sigma$ and $\Gamma$.
Throughout the article, we assume that they are nonunary. A function $h: \Sigma^* \to \Gamma$ is a \emph{morphism} if $h(uv)=h(u)h(v)$ for all words $u,v\in \Sigma^*$. We only give the images of letters when defining morphisms as those determine images of all words.

Let $\inj$ be the set of all injective morphisms $\Sigma^* \to \Gamma^*$.

Consider the following variations of the different exponents that were defined above:
\begin{align*}
    \iei(u) &= \sup \{\ie(h(u)) \mid h \in \inj\},\\
    \fei(u) &= \sup \{\fe(h(u)) \mid h \in \inj\},\\
    \cei(w) &= \sup \{\ce(h(w)) \mid h \in \inj\},\\
    \acei(w) &= \sup \{\ace(h(w)) \mid h \in \inj\}.
\end{align*}
Of these, $\iei$ has previously been studied in \cite{Saarelapowers} by Saarela,
and $\cei$ is not interesting because it is infinite for all $w$.
In this article, we study $\fei$ and $\acei$, which we call \emph{mapped (fractional) exponent} and \emph{mapped asymptotic critical exponent}, respectively. Note that these do not depend on the choice of $\Gamma$ as long as $\abs{\Gamma}\geq 2$. This is because if $\Gamma=\{c_1,c_2,c_3,\ldots, c_n\}$ we can define morphism $h$ by $h(c_i)=c_1^{n+1-i}c_2^{i}$ and by observation in \cite{cassaigne2008extremal}, for any $h'\in \inj$ and $w\in \Sigma^*$ we have $\fe(h'(w))\leq \fe(h(h'(w))$ and $h \circ h' \in \inj$. Thus we could use only two letters of $\Gamma$ when finding the different suprema above.

While we concentrate on $\fei$ and $\acei$, the following variation of $\acei$ could also be considered:

\begin{displaymath}
    \aceii(w) = \limsup_{n \to \infty} \{\fei(u) \mid u \in \fact_n(w)\}.
\end{displaymath}
However, in a vacuum, we consider $\acei$ to be more interesting of the two. We will demonstrate this in some examples in this paper.

\subsubsection{Compactness properties of sequences}

In this subsection,
we give formal definitions for different kinds of finite and infinite sequences.
Finite and infinite words, of which we already talked about in the previous subsection,
can then be defined as sequences over a finite alphabet.
We also describe how to define infinite sequences as limits.
Most of the content of this subsection is only used in a couple of specific places in this article,
namely, in Lemma \ref{lemma:unbounded-repetitions-of-x-in-image}
and Lemma \ref{lemma:bi-infinite2factorizations}.

We need four different types of sequences:
finite, shifted finite, right-infinite and bi-infinite.
Formally, sequences over a set $X$ are defined as functions $f: A \to X$ in the following way:
\begin{itemize}
    \item
    If $A = \{0, \dots, n\}$ for some $n \in \N$,
    then $f$ is a \emph{finite sequence}
    and we write
    \begin{displaymath}
        f = (f(0), \dots, f(n)).
    \end{displaymath}

    \item
    If $A = \{-m, \dots, n\}$ for some $m, n \in \N$,
    then $f$ is a \emph{shifted finite sequence}
    and we write
    \begin{displaymath}
        f = (f(-m), \dots, f(-1) \sep f(0), \dots, f(n)).
    \end{displaymath}
    \item
    If $A = \N$, then $f$ is a \emph{right-infinite sequence} (or just \emph{infinite sequence})
    and we write
    \begin{displaymath}
        f = (f(0), f(1), f(2), \dots).
    \end{displaymath}
    \item
    If $A = \Z$, then $f$ is a \emph{bi-infinite sequence}
    and we write
    \begin{displaymath}
        f = (\dots, f(-2), f(-1) \sep f(0), f(1), f(2), \dots).
    \end{displaymath}
\end{itemize}
Instead of $f(i)$ we often write $f_i$. We can also write $f = (f(i))_{i \in A}$ or $f = (f(i))$ if $A$ is clear from the context. The \emph{length} of $f$ is $|f| = \abs{A}$, where it naturally might be that $\abs{f}=\infty$. By $\abs{f}_x$ where $x\in X$ we mean the size of the set $\{i \in A \mid f(i)=x\}$, that is, the number of elements $x$ in the sequence $f$.
The set of all functions $A \to X$ is denoted by $X^A$. The definition of shifted finite sequences is nonstandard
and is only used for the purpose of constructing bi-infinite sequences as limits.

If $f$ is a right-infinite sequence, $t$ a bi-infinite sequence and $g: \Z \to \Z$ is strictly increasing,
then 
\begin{displaymath}
  f \circ g = (f(g(0)), f(g(1)), f(g(2)), \dots)  
\end{displaymath}
 is a \emph{subsequence} of $f$, and
 \begin{displaymath}
   t \circ g = (\dots, t(g(-2)), t(g(-1)) \sep t(g(0)), t(g(1)), t(g(2)), \dots)  
 \end{displaymath} is a subsequence of $t$.
More specifically, these can be called the \emph{$g$-subsequences} of $f$ and $t$ respectively. Similarly if $w$ is a finite of shifted finite sequence and $g:B \to A$ is strictly increasing and such that $w\circ g$ is also a finite or a shifted finite sequence, then $w\circ g$ is called a ($g$-)subsequence of $w$.

Instead of using the sequence notation that was defined above,
words (sequences over finite $X$) can of course be written in a more usual way as follows:
\begin{align*}
    (a_0, \dots, a_n) &= a_0 \dotsm a_n, \\
    (a_{-m}, \dots, a_{-1} \sep a_0, \dots, a_n)
        &= a_{-m} \dotsm a_{-1} \sep a_0 \dotsm a_n, \\
    (a_0, a_1, a_2, \dots) &= a_0 a_1 a_2 \dotsm, \\
    (\dots, a_{-2}, a_{-1} \sep a_0, a_1, a_2, \dots)
        &= \dotsm a_{-2} a_{-1} \sep a_0 a_1 a_2 \dotsm.
\end{align*}
The set of all words that are finite sequences over $X$ is denoted by $X^*$. This includes empty word, that is the unique function $\eps:\emptyset \to X$.
The basic concepts on words, such as factor, prefix, suffix and morphism, naturally work on sequences that are words.

A sequence $(u_i)_{i \in \N}$ of finite or infinite sequences over $X$
\emph{converges} to a right-infinite sequence $w$
(or, in other words, has a \emph{limit} $w$)
if for all $k \in \N$ there exists a bound $L_k$ such that $w(k) = u_i(k)$ for all $i > L_k$.
It is possible to turn the set of all finite and right-infinite sequences into a metric space
so that the above definition of a limit
is compatible with the definition of a limit in the metric space,
see, e.g., \cite{Lothairealg}.
The following well-known lemma is a consequence of the fact that this metric space is compact
if $X$ is finite.

\begin{lemma}[One of the problems in the first chapter of  \cite{Lothairealg}, related to Theorem 1.2.4] \label{lem:compact-right}
     Let $u_i$ be a finite or an infinite sequence over a finite set $X$ for all $i\in \N$. If $\lim_{i \to \infty} |u_i| = \infty$,
    then the sequence $(u_i)_{i \in \N}$ has a subsequence that converges to a right-infinite sequence.
\end{lemma}

Similarly, a sequence $(u_i)_{i \in \N}$ of shifted finite or bi-infinite sequences over $X$
\emph{converges} to a bi-infinite sequence $w$
(or, in other words, has a \emph{limit} $w$)
if for all $k \in \Z$ there exists a bound $L$ such that $w(k) = u_i(k)$ for all $i \geq L$.
The following lemma is analogous to Lemma~\ref{lem:compact-right}, and the proof is essentially the same.

\begin{lemma} \label{lem:compact-bi}
    Let $u_i: A_i \to X$ be shifted finite or a bi-infinite sequences over a finite set $X$
    and let $\lim_{i \to \infty} \min A_i = -\infty$
    and $\lim_{i \to \infty} \max A_i = \infty$.
    The sequence $(u_i)_{i \in \N}$ has a subsequence that converges to a bi-infinite sequence.
\end{lemma}

\section{Finite words}

The main theorems of this section are the classification of words $w \in \Sigma^+$ such that $\fei(w)=\infty$ and its binary restriction. These are Theorem \ref{thm:fei-infty} and Corollary \ref{thm: binaryfei}. There are also noteworthy example words of low mapped exponent and high but finite mapped exponent in Theorems \ref{thm: lowpower} and \ref{thm: highpower}.

\subsection{Lemmas for finite words}

We present some needed results for the study of mapped exponent. The first result is the famous theorem of Fine and Wilf.

\begin{theorem}[Fine and Wilf \cite{fine1965uniqueness}]
    \label{Fine and Wilf}
    If a power of $u$ and a power of $v$
    have a common prefix of length $|u| + |v| - \gcd(|u|, |v|)$,
    then $u$ and $v$ are integer powers of a common word.
\end{theorem}

The following is a folklore lemma that we prove for completeness.

\begin{lemma}
\label{lemma: only-trivial-factorization}
    For any rational $r$ and primitive word $x$, $x$ can be a factor of $x^r$ only in a trivial way, that is, if $x^r = txu$, then $t \in x^*$.
\end{lemma}
\begin{proof}
    Let $n$ be an integer bigger than $r$. If $x$ can be a factor of $x^n$ only in a trivial way, then the same is true for $x^r$. Assume for a contrary that $x$ is a factor of $x^n$ in a nontrivial way. Then $x$ is a nontrivial factor of $x^2$, that is, $xx = txu$ for some nonempty words $t$ and $u$. We see that $t$ is a prefix of $x$ and $u$ is a suffix of $x$. We also see that $\abs{tu}=x$, so together we know that $x=tu$. We get that $xx=tutu=txu=ttuu$ and then by canceling we have $ut=tu$. Thus $u=z^k$ and $t=z^l$ for some word $z$ and positive integers $k$ and $l$ and then $x=z^{k+l}$, contradicting the fact that $x$ is primitive.
 \end{proof}

\begin{lemma}
    \label{lemma:conj-in-power}
    For any rational $r$ and primitive word $x = pq$,
    if $x^r = tqpu$, then $t \in (pq)^* p$ and $u = (qp)^{r'}$ for some rational $r'$.
\end{lemma}

\begin{proof}
    We have $qtqpu = q x^r = q (pq)^r = (qp)^\gamma$ for some rational $\gamma$.
    By Lemma~\ref{lemma: only-trivial-factorization},
    $qt = (qp)^n$ for some positive integer $n$ and thus $t = (pq)^{n - 1} p$,
    and then $u = (qp)^{\gamma - n - 1}$.
\end{proof}

\begin{lemma}
   \label{lemma:bouned-letter}
   Let $w \in \Sigma^+$, $h \in \mathcal{I}$ and $h(w)=x^r$ for $x\in \Gamma^+$ of minimal length and $r\in \Q$. If $\abs{h(a)}\geq \abs{x}$ for some $a \in \Sigma$ and $ava$, $aua \in \fact_+(w)$ for some words $v\in (\Sigma \setminus\{a\})^*$ and $u\in (\Sigma \setminus\{a\})^*$, then $v=u$.
\end{lemma}

\begin{proof}
    We can write $w = u_1 aua u_2 = v_1 ava v_2$ for some words $u_1, u_2, v_1, v_2 \in \Sigma^*$.
    Because $h(a)$ is a factor of $x^r$,
    there must be $p, q$ such that $x = pq$ and $h(a)$ begins with $qp$ (that is, a prefix of $h(a)$ is a conjugate of $X$).
    By Lemma~\ref{lemma:conj-in-power},
    $h(u_1), h(u_1 au), h(v_1), h(v_1 av) \in (pq)^* p$, and thus $h(au), h(av) \in (qp)^*$.
    It follows that $h(auav) = h(avau)$ and, by the injectivity of $h$, $auav = avau$.
    Since $a$ does not occur in $u$ or $v$, this leads to $u = v$.
\end{proof}

Part of the next lemma was already seen in \cite{Saarelapowers}.

\begin{lemma}
   Let $w\in \Sigma^*$ be a word. If there exists an injective morphism $h\in \inj$ such that $h(w)=x^r$ for some $x \in \Gamma^*$ and $r\geq 1$ and where $|x|_a=1$ for some $a\in\Gamma$, then $\fei(w)=\infty$.
   \label{Lemma: single letter unbounded}
\end{lemma}

\begin{proof}
Since $\abs{x}_a=1$, we can write $x=uav$ for some words $u,v\in (\Gamma\setminus\{a\})^*$. Then $h(w)=(uav)^r=uavuav\cdots uav u'$, where $u'$ is some prefix of $uav$. For all $n\in\mathbb{N}$, we can define an injective morphism $h': \Gamma^* \to \Gamma^*$ by $h'(a)=a(vua)^{n-1}$ and $h'(b)=b$ for all letters $b\in\Gamma\setminus\{a\}$. We then obtain:
\begin{equation*}
\begin{split}
    h'(h(w))&=ua(vua)^{n-1}vua(vua)^{n-1}v\dots ua(vua)^{n-1}v h'(u'),\\
    &=(uav)^{n \lfloor r \rfloor }h'(u').
\end{split}
\end{equation*}
It then remains to show that $h'(u')=(uav)^{r'}$, where $r'\in\mathbb{Q}$.
\begin{itemize}
    \item If $a\notin \alphabet(u')$, then $h'(u')=u'$, and as $u'$ is a prefix of $uav$, it follows that $h'(u')$ is a prefix of $uav$.
    \item If $a\in \alphabet(u')$, then $u'=uam$, where $m\in(\Gamma\setminus\{a\})^*$ is a prefix of $v$. Then,
    \begin{equation*}
    \begin{split}
    h(u')&=ua(vua)^{n-1}m,\\
    &=ua(vua)^{n-2}vuam,\\
    &=(uav)^{n-1}uam.
    \end{split}
    \end{equation*}
    As $m$ is a prefix of $v$, it follows that $uam$ is a prefix of $uav$.
    \end{itemize}
\end{proof}

\subsection{Mapped fractional exponent}

In this subsection, we give a characterization for the words that can be mapped to arbitrarily high fractional powers by injective morphisms. In the binary case, the characterization reduces to a very easy-to-check statement. For words $w \in \Sigma^+$ with finite $\fei(w)$, we prove the bound $\fei(w)\leq \abs{w}$,
while also providing words $w_n$ (over increasingly large alphabets) with $\fei(w_n) = O(\abs{w_n})$.

While computing $\fei$ is difficult for most words, there are some words for which it is relatively easy. We demonstrate this with a small example. Let $a,b\in \Sigma$ and consider a word $(aabb)^2 \in \Sigma^*$ and morphism $h\in \inj$. We know $h(aabb)$ is primitive (see \cite{Saarelapowers}). Let $h((aabb)^2) = x^r$ for some primitive word $x\in \Gamma^*$ and $r\geq2$. Since a power of $h(aabb)$ and a power of $x$ have a common prefix $h((aabb)^2)$, by Theorem of Fine and Wilf \ref{Fine and Wilf} we see that $x$ and  $h(aabb)$ are integer powers of a common word. We can apply Theorem of Fine and Wilf since  $\abs{h((aabb)^2)} \geq \abs{h(aabb)x}$, which holds since $r\geq 2$. By primitivity this leads to $x=h(aabb)$ and thus $r=2$. Since $h\in \inj$ was arbitrary, we see that $\fei((aabb)^2)=\fe((aabb)^2)=2$. The next theorem shows that $\fei(w)$ can be arbitrarily close to 1 using similar ideas.

\begin{theorem}
\label{thm: lowpower}
    Let $a,b \in \Sigma$ be different letters. For all integers $n \geq 2$, $\fei((ab)^n ba) = 1 + \frac{2}{n-1}$.
\end{theorem}
\begin{proof}
    Let $w = (ab)^n ba$ with $n \geq 2$.
    Let $h$ be an injective morphism
    and let $h(w) = v^r$ for some word $v$ and rational number $r$.
    We first prove that $r \leq 1 + \frac{2}{n-1}$. We may assume $v$ to be primitive.

    A power of $h(ab)$ and a power of $v$ have a common prefix of length $|h(ab)^n|$.
    If this is at least $|h(ab)v|$,
    it follows from Theorem of Fine and Wilf $\ref{Fine and Wilf}$
    that $v$ and $h(ab)$ are powers of a common word.
    Since $v$ is primitive this leads to $h(ab)=v^m$ for some $m\in \N$. Now by length argument, $h(ba)=v^m=h(ab)$, which contradicts the injectivity of $h$.
    This shows that $|h(ab)^n| < |h(ab)v|$ and thus $|v| > |h(ab)^{n - 1}|$.
    Consequently,
    \begin{displaymath}
        r = \frac{|h(w)|}{|v|}
        < \frac{n + 1}{n - 1} = 1 + \frac{2}{n - 1}.
    \end{displaymath}
    As $h\in \inj$ was arbitrary, $\fei(w)\leq 1+\frac{2}{n-1}$.

    For the converse, let $h_k \in \inj$ be defined by $h_k(a)=(cd)^k c$ and $h_k(b)=dc$ for each integer $k\in \N$ and for some different letters $c,d \in \Gamma$.
    Then
    \begin{align*}
      h_k(w)&= ((cd)^k c dc)^n dc (cd)^k c  \\
      &=  ((cd)^k c dc)^{n-1} (cd)^k c dc dc (cd)^k c  \\
      &= [((cd)^k c dc)^{n-1} cd ] (cd)^k c dc (cd)^k c \\
      &= [((cd)^k c dc)^{n-1} cd ]^{1+\frac{4k+4}{(2k+3)(n-1)+2}},
    \end{align*} where the square brackets show a period of $h_k(w)$.
    So 
    \begin{displaymath}
      \fei(w)\geq \lim_{k \to \infty} \fe(h_k(w)) \geq \lim_{k \to \infty} 1+\frac{4k+4}{(n-1)(2k+3)+2}=1+\frac{2}{n-1}.  
    \end{displaymath}
    \end{proof}

We say two words $u$ and $v$ are \emph{prefix-comparable} if there exists a word $s$ such that $u$ is a prefix of the word $vs$. Similarly the words $u$ and $v$ are \emph{suffix-comparable} if there exists a word $p$ such that $u$ is a suffix of $pv$.

\begin{theorem} \label{thm:fei-infty}
    For a word $w \in \Sigma^+$, the following are equivalent:
    \begin{enumerate}
        \item
        $\fei(w) = \infty$.
        \item
        $\fei(w) > |w|$.
        \item
        There exists an injective morphism $h\in \inj$ and a letter $a\in \alphabet(w)$
        such that $h(w)=x^r$ for some primitive $x\in \Gamma^*$,
        $r \in \Q$ and $\abs{h(a)} \geq \abs{x}$.
        \item
        There exists an integer $k \geq 0$, a letter $a \in \Sigma$,
        words $w_1, w_2, w_3 \in (\Sigma \smallsetminus \{a\})^*$,
        and an injective morphism $h$
        such that $w = w_1 (a w_2)^k a w_3$,
        $h(w_1)$ and $h(w_2)$ are suffix-comparable,
        and $h(w_2)$ and $h(w_3)$ are prefix-comparable.
    \end{enumerate}
\end{theorem}
\begin{proof}
    We prove the claim as an implication chain $1. \implies 2. \implies 3. \implies 4. \implies 1.$. Clearly $1. \implies 2.$, so we can start at $2. \implies 3.$.

    Assume $w \in \Sigma^+$ has $\fei(w) > |w|$. By definition, there exists an injective morphism $h\in \inj$ such that $h(w)=x^r$ with $r \geq |w|$. If for all $a \in \alphabet(w)$ we have $\abs{h(a)}< \abs{x}$, then $\abs{h(w)}<\abs{w}\abs{x}\leq \abs{x^r}$, a contradiction. Thus $2. \implies 3.$.

    Next, $3. \implies 4.$.
    By Lemma \ref{lemma:bouned-letter},
    between consecutive appearances of $a$ in $w$ we always have the same word.
    This happens exactly when $w$ admits the factorization $w=w_1 (a w_2)^k a w_3$
    for words $w_1,w_2,w_3\in (\Sigma \setminus\{a\})^*$ and integer $k\in \N$.
    If $k=0$, we can freely choose $w_2 = \eps$ so the comparability claim of $4.$ holds.
    If $k>0$, then since $\abs{h(a)}\geq \abs{x}$ and $h(a)$ is a factor of a repetition of $x$,
    there are $p, q\in \Gamma^*$ such that $x = pq$ and $h(a)$ begins with $qp$ (that is, a prefix of $h(a)$ is a conjugate of $x$).
    By Lemma \ref{lemma:conj-in-power}, $h(w_1), h(w_1 a w_2) \in (pq)^* p$.
    Thus $h(w_1)$ and $h(w_2)$ are suffix-comparable.
    Similarly, by Lemma \ref{lemma:conj-in-power},
    $h(a w_2 a w_3)$ and $h(a w_3)$ are both prefixes of some word in $(qp)^*$,
    and thus $h(w_2)$ and $h(w_3)$ are prefix-comparable.

    Lastly we show that $4. \implies 1.$. We assume a factorization $w= w_1 (a w_2)^k a w_3$ with $k \in \N$, words $w_1,w_2,w_3\in (\Sigma \setminus\{a\})^*$ and an injective morphism $h$ such that $h(w_1)$ and $h(w_2)$ are suffix-comparable and $h(w_2)$ and $h(w_3)$ are prefix-comparable. By definition, there exists words $p, s$ such that $h(w_1)$ is a suffix of $ph(w_2)$ and $h(w_3)$ is a prefix of $h(w_2)s$. Define an injective morphism $h'$ with $h'(b)=h(b)$ for all $b \neq a$ and $h'(a)= s c p$ for some new letter $c \not \in \Gamma$. Now, since $\abs{w_1}_a=\abs{w_2}_a=\abs{w_3}_a=0$, we have 
    \begin{displaymath}
    h'(w) = h(w_1)s (c ph(w_2)s)^k c p h(w_3).    
    \end{displaymath}
    By suffix-comparability, $ph(w_2) = t h(w_1)$ for some word $t$, 
    \begin{displaymath}
     h'(w)=(h(w_1)sct)^k h(w_1) s c p h(w_3).   
    \end{displaymath}
    Similarly by prefix-comparability $ph(w_3)$ is a prefix of $ph(w_2)s=th(w_1)s$, finally we have 
    \begin{displaymath}
     h'(w)=(h(w_1)sct)^r   
    \end{displaymath}
    for some $r<k+2$. Now $\abs{h(w_1)sct}_c = 1$, so by Lemma $\ref{Lemma: single letter unbounded}$ it follows that $\fei(w)=\infty$ (by using $\Gamma \cup\{c\}$ as $\Sigma$). This completes the implication chain.
\end{proof}

In the binary case the characterization becomes simpler.
\begin{corollary}
\label{thm: binaryfei}
    For a binary word $w \in \{a, b\}^+$, where $a,b\in \Sigma$, the following are equivalent:
    \begin{enumerate}
        \item
        $\fei(w) = \infty$.
        \item
        $\fei(w) > \abs{w}$.
        \item
        There exist integers $k, j_1, j_2, j_3 \geq 0$
        such that $w = b^{j_1} (a b^{j_2})^k a b^{j_3}$ \\
        or $w = a^{j_1} (b a^{j_2})^k b a^{j_3}$.
    \end{enumerate}
\end{corollary}

\begin{proof}
    Follows from Theorem $\ref{thm:fei-infty}$ from statements $1$, $2$ and $4$.

\end{proof}

However, in general, the factorization $w=w_1(a w_2)^k a w_3$ is not enough to ensure unbounded $\fei(w)$. This can be seen with the word $w=bbccabcbca$, where $w_1=bbcc$ and $w_2=cbcb$ cannot be injectively mapped to suffix-comparable words.

Lastly, the next theorem shows that $\fei(w)$ can be large without being infinite for a word with low $\fe(w)$.

\begin{theorem}
\label{thm: highpower}
    For all $n$, there exists a $2n$-ary word $w_n$ of length $6n$ with $\fe(w_n)=1$ such that
    \begin{displaymath}
        n - \frac{1}{6} < \fei(w_n) < \infty.
    \end{displaymath}
\end{theorem}

\begin{proof}
    For $n=1$, the lower bound is obtained by identity morphism and it is enough to note that $w_1 = aababb$ is a binary word which by Theorem $\ref{thm:fei-infty}$ has $\fei(w_1)< \infty$. So let $n\geq 2$.
    Let $a_1, \dots, a_n, b_1, \dots b_n$ be distinct letters in $\Sigma$ and let
    \begin{displaymath}
        w_n = \prod_{i = 1}^n (a_i a_i b_i a_i b_i b_i).
    \end{displaymath}
    Since $w_n$ does not admit a factorization as in fourth statement in Theorem~\ref{thm:fei-infty}, $\fei(w_n) < \infty$.
    Let $h\in \inj$ be the morphism defined by
    \begin{displaymath}
        h(a_i) = c^{i - 1} a c^{n - i},\qquad
        h(b_i) = c^{i - 1} b c^{n - i}
    \end{displaymath}
    for all $i\leq n$ and for some distinct letters $a,b,c \in \Gamma$.

    We have
    \begin{displaymath}
        h(w_n) =
        (a c^{n - 1} a c^{n - 1} b c^{n - 1} a c^{n - 1} b c^{n - 1} b c^n)^{n - \frac{n}{6n+1}},
    \end{displaymath}
    so $\fei(w_n) \geq n - \frac{n}{6n+1} > n - \frac{1}{6}$.
\end{proof}

\section{Infinite words}
We now change our focus to mapped asymptotic critical exponent and infinite words. Very similarly to the case of finite words and mapped exponent, the binary case provides more powerful results. However, while the binary case of $\fei$ is really just a restriction of the general result, the same is not true for $\acei$. Thus we split this section on $\acei$ and infinite words into four subsections, including the subsection of lemmas below. The other three are the basic properties of $\acei$ over general alphabet, a special case of synchronizing words (which we define in that section) and finally the binary case of $\acei$.

\subsection{Lemmas for infinite words}
In this section we present and prove some needed results about injective morphisms, $\ace$ and periodic words.
\begin{lemma}
\label{lemma: inside-h}
    For all words $w\in \Sigma^\N$ and injective morphisms $h\in \inj$,
    \begin{displaymath}
        \ace(h(w))=\limsup_{n \to \infty} \{ \mathrm{E}(f) \ \vert \ f \in \mathrm{Fact}_n(h(w)) \} = \limsup_{n \to \infty} \{ \mathrm{E}(h(f)) \ \vert \ f \in \mathrm{Fact}_n(w) \}.
    \end{displaymath}
    
\end{lemma}
\begin{proof}
    Let
    \begin{align*}
      S_1 & = \limsup_{n \to \infty} \{ \mathrm{E}(f) \ \vert \ f \in \mathrm{Fact}_n(h(w)) \} \text{ and } \\
      S_2 & =\limsup_{n \to \infty} \{ \mathrm{E}(h(f)) \ \vert \ f \in \mathrm{Fact}_n(w) \}
    \end{align*}
    to ease up the notation. Clearly $S_2 \leq S_1$.

    Assume that $S_1>S_2 + \lambda$ for some positive real number $\lambda$. Then:

    \begin{itemize}
        \item By the definition of $S_1$, for each natural number $k$, there exists a factor $u_k=x_k^r$ of $h(w)$ such that $r \geq S_2 + \lambda$ and $\abs{u_k} \geq k$.

        \item By the definition of $S_2$, there exists a natural number $N$ such that for all $f \in \mathrm{Fact}_k(w)$ where $k \geq N$, we have $\mathrm{E}(h(f)) < S_2 + \frac{\lambda}{2}$.

    \end{itemize}

    Note that for each $k$ we can assume $r$ above to satisfy $r < S_2 + \lambda +1$ since if there are arbitrarily long factors of the type $x_k^r$ where $r \geq 2$, then there will be arbitrarily long factors of the type $x_k^{r -1}$. Now since $r$ is bounded, we have a lower bound for $\abs{x_k}$ in terms of $k$:
    \begin{displaymath}
        \abs{x_k} \geq \frac{k}{r} > \frac{k}{S_2 + \lambda +1}.
    \end{displaymath}

    Let $m=\max\{\abs{h(a)} \mid a \in \alphabet(w) \}$. Clearly within $w$ there is a factor $u'_k$ such that $u_k = p_k h(u'_k) s_k$ for some words $p$ and $s$ such that $\abs{p_k}<m$ and $\abs{s_k}<m$.

    Now let $k$ be so large that $\frac{\lambda}{2} \abs{x_k} > \max \{ 2m, N m \}$. Then $u_k=x_k^r = p_k h(u_k') s_k$ where $\abs{p_k s_k} < 2m < \frac{\lambda}{2} \abs{x_k}$. So 
    \begin{displaymath}
        \abs{h(u_k')}= \abs{p_k h(u_k') s_k }- \abs{p_k s_k} > (S_2 + \frac{\lambda}{2})\abs{x_k}.
    \end{displaymath}
    Since $h(u_k')$ is a factor of $x_k^r$ and of size larger than $\abs{x_k}$, $h(u_k') = y_k^{r'}$ for some conjugate $y_k$ of $x_k$. By the size of $h(u_k')$ we must have $r' > S_2 + \frac{\lambda}{2}$ and on the other hand we have $\abs{h(u_k')}>N m$, which means that $\abs{u_k'}>N$. This is a contradiction with the definition of $N$ and $S_2$.

    So the original assumption of $S_1>S_2 + \lambda$ is wrong for all positive real numbers $\lambda$, so it must be that $S_1 = S_2$.
\end{proof}

Infinite word $w$ is said to be \emph{eventually periodic} if $w=puuuu\dotsm$ for some non-empty word $u$. If $p$ is empty, then $w$ is said to be \emph{periodic}.
Similarly, if $w$ is a bi-infinite word, then it is said to be periodic if $w = \dotsm uuu\dotsm$ for some non-empty word $u$.

If $w$ is infinite word that is not eventually periodic or bi-infinite word that is not periodic, it is said to be \emph{aperiodic}.

\begin{theorem}[Morse and Hedlund \cite{morsehedlund}]
\label{thm:morsehedlund}
    For all infinite words $w$, let $\rho_n(w)=\abs{\fact_n(w)}$. The following are equivalent:
    \begin{enumerate}
        \item $w$ is eventually periodic.
        \item $\rho_n(w)\leq C$ for some constant  $C \in \R$ for all $n\in \N$.
        \item $\rho_n(w) \leq n$ for some $n\in \N$.
    \end{enumerate}

\end{theorem}

\begin{lemma}
\label{lemma:eventuallyperiodic}
    Let $h \in \mathcal{I}$ and $u \in \Sigma^\N$ be an infinite word. If $h(u)$ is eventually periodic, then so is $u$.
\end{lemma}
\begin{proof}
    Assume $h(u)$ is eventually periodic as in the statement of the theorem.
    Let $F_n= \{h(f) \mid f \in \fact_+(u) \text{ and } \abs{f}\leq n \}$. The words of $F_n$ are factors of $h(u)$ and have length at most $n \max \{ \abs{h(a)} \} $, where $a$ is a letter. Then by Theorem $\ref{thm:morsehedlund}$ we have $\abs{F_n} \leq Cn\max\{ \abs{h(a)} \} $ for some constant $C\in \R$. On the other hand, since $h$ is injective, each $f \in \fact_+(u)$ with $\abs{f}\leq n$ gives a unique element to $F_n$, and thus, using the notation of Theorem $\ref{thm:morsehedlund}$, $\abs{F_n}=  \sum_{i=1}^n\rho_i(u)$. If $u$ is aperiodic, then by Theorem $\ref{thm:morsehedlund}$ we would have $\rho_i(u)\geq i+1$, which would lead to $\abs{F_n}$ growing faster than linearly, a contradiction. Thus $u$ is eventually periodic.
\end{proof}

\begin{lemma}
    \label{lemma:unbounded-repetitions-of-x-in-image}
    Let $w \in \Sigma^\N$ be an infinite word and $h\in \mathcal{I}$. If there exists a word $x\neq \varepsilon$ such that $x^l$ is a factor of $h(w)$ for all $l\in \N$, then $\ace(w)=\infty$.
\end{lemma}
\begin{proof}
    Recall the definition of limits of words and subsequences of type $w \circ g$ defined in the preliminaries. Let $\sigma: \Z \to \Z$ be the shift map defined by $\sigma(n)=n+1$. For $\sigma$ and its iterates $\sigma^n$, we write (as is more common) $\sigma^n(f) \coloneq f \circ \sigma^n$.
    Since $x^l$ is a factor of $h(w)$ for all $l\in \N$, there exists a sequence of natural numbers $t_1, t_2, \ldots$ such that $\lim_{i \to \infty} \sigma^{t_i}(h(w)) = x^\omega$. Let $M = \max \{ \abs{h(a)} \vert \ a \in \alphabet(w) \}$. Then for each $i$, there exists $p_i \leq M$ such that $\sigma^{t_i+p_i}(h(w)) = h(\sigma^{t'_i}(w))$ for some $t'_i$. Since the $p_i$ are bounded, we have an infinite sequence of natural numbers $j_i$ such that $p_{j_i}=p_{j_1}$. Finally, by possibly considering a subsequence of $j_i$, we may assume that $\lim_{i \to \infty} \sigma^{t'_{j_i}}(w)$ exists by \ref{lem:compact-right}. Now
    \begin{align*}
    & \sigma^{p_{j_1}}(\lim_{i \to \infty} \sigma^{t_{j_i}}(h(w))) \\
    =& \lim_{i \to \infty} \sigma^{t_{j_i}+p_{j_1}}(h(w)) \\
    =& \lim_{i \to \infty} \sigma^{t_{j_i}+p_{j_i}}(h(w)) \\
    =& \lim_{i \to \infty} h(\sigma^{t'_{j_i}}(w)) \\
    =& h(\lim_{i \to \infty} \sigma^{t'_{j_i}}(w))
    \end{align*}
    where the last equality is by the continuity of $h$ in the metric space of sequences, or by following: Let $a\in \Sigma$ be the first letter in the infinite word $\lim_{i \to \infty} \sigma^{t'_{j_i}}(w)$. If $h(a)$ is then not a prefix of $\lim_{i \to \infty} h(\sigma^{t'_{j_i}}(w))$, it must be that infinite number of the words $\sigma^{t'_{j_i}}(w)$ begin with a letter $b\neq a$. This contradicts the assumption that $a$ is a first letter of $\lim_{i \to \infty} \sigma^{t'_{j_i}}(w)$. Similar argument works on arbitrary large prefixes of $\lim_{i \to \infty} \sigma^{t'_{j_i}}(w)$.
    
    Now we see that a sequence of shifts of $w$ has a limit point that is a preimage of a periodic point. By Lemma $\ref{lemma:eventuallyperiodic}$, we have that this preimage must be also eventually periodic. But since this preimage limit point has only factors that are also factors of $w$ (as the limit is over shifted versions of $w$), we see that $\ace(w)=\infty$.
\end{proof}

Let $X \subseteq \Sigma^*$ be a finite set of finite words.
If $w: \Z \to \Sigma$ is bi-infinite sequence (written here as a concatenation of words)
\begin{displaymath}
    w = \dotsm x_{-2} x_{-1} p \sep q x_1 x_2 \dotsm,
\end{displaymath}
where $x_i \in X$ for all $i$, $pq \in X$, $q \ne \eps$, then
\begin{displaymath}
    (\dots, x_{-2}, x_{-1}, p \sep q, x_1, x_2, \dots)
\end{displaymath}
is an \emph{$X$-factorization} of $w$.

\emph{The combinatorial rank} of a finite set of words $X$ is defined as the size of the smallest set $Y$ such that $X \subseteq Y^*$.

\begin{theorem}[Karhum{\"a}ki, Ma{\v{n}}uch and Plandowski \cite{karhumaki2003defect}]
\label{thm:defect-theorem}
    Let $X$ be a finite set of nonempty words and let $w$ be a bi-infinite word.
    If $w$ has two different $X$-factorizations, then $w$ is periodic or the combinatorial rank of $X$ is at most $\abs{X}-1$.
\end{theorem}

\begin{corollary}
    \label{lemma:bi-infinite-periodic-word}
    Let $h: \{a,b \}^* \to \{a,b \}^*$ be an injective morphism and let $w \in \{a,b\}^\Z$ be a bi-infinite word. If $w$ can be $\{h(a),h(b)\}$-factored in two different ways, then $w$ is periodic.
\end{corollary}
\begin{proof}
    By Theorem $\ref{thm:defect-theorem}$, since $w$ can be $\{h(a),h(b)\}$-factored in two different ways, either $w$ is periodic, or the combinatorial rank of $\{h(a),h(b)\}$ is $1$. Since $h$ is injective, the later is not the case and thus $w$ is periodic.
\end{proof}

An \emph{X-interpretation} of a finite word $w \in \Sigma^*$ is a factorization $w=w_1 w_2 w_3 \cdots w_n$ such that $w_1$ is a suffix of a word in $X$, $w_n$ is a prefix of a word in $X$ and $w_i \in X$ for all $1<i<n$.
Two $X$-interpretations $(w_1, w_2, w_3, \dots, w_n)$ and $(v_1, v_2, v_3, \dots, v_m)$
of the same word are \emph{disjoint}
if $w_1w_2\cdots w_i \neq v_1v_2\cdots v_j$ for all $i<n$ and $j<m$.
If two $X$-interpretations are not disjoint, we say that they are \emph{linked}. The \emph{$X$-degree} of a word $w$
is the maximal number of pairwise disjoint $X$-interpretations $w$ has.

\begin{theorem}[Theorem 8.3.1 in \cite{Lothaire}]
\label{lothaire 8.3.1}
    Let $w\in \Sigma^*$ be a word and $X \subseteq \Sigma^*$ a finite set of words. Let $r = \mathrm{E}(w)$ and $w=x^r$.
    If $\abs{x}>\max \{ \abs{v} \vert \ v \in X \}$, then the $X$-degree of $w$ is at most $\abs{X}$.
\end{theorem}

\subsection{Basic properties of mapped asymptotic critical exponent}

Recall the definition of mapped asymptotic critical exponent of $w \in \Sigma^\N$: 
\begin{displaymath}
    \acei(w) = \sup \{\ace(h(w)) \mid h \in \inj\} = \sup_{h\in \inj} \big( \limsup_{n \to \infty} \{ \mathrm{E}(f) \ \vert \ f \in \mathrm{Fact}_n(h(w)) \} \big).
\end{displaymath}

In this subsection, we prove some expected properties of this definition. We also compare this definition to the alternative definition $\aceii$ introduced earlier.

In the definition of mapped asymptotic critical exponent, we decided on taking the exponent of factors of the word $h(w)$. Another way to define $\acei$ would have been to consider exponents of $h(f)$, where $f$ is a factor of $w$. By Lemma $\ref{lemma: inside-h}$, this ends up not making a difference. Theorem $\ref{thm:general-upperbound}$ shows that for all words $w$ with finite $\ace(w)$, the mapped asymptotic critical exponent is also finite.

\begin{theorem}
\label{thm: inside-h}
    For all words $w \in \Sigma^\N$, 
    \begin{displaymath}
        \mathrm{ACE}_\mathcal{I} (w) = \sup_{h\in \inj} \big( \limsup_{n \to \infty} \{ \mathrm{E}(h(f)) \ \vert \ f \in \mathrm{Fact}_n(w) \} \big).
    \end{displaymath}
\end{theorem}
\begin{proof}
    Follows directly from Lemma $\ref{lemma: inside-h}$.
\end{proof}

\begin{theorem}
\label{thm: prefix dont matter ACE}
Let $w\in \Sigma^\N$ and $u\in \Sigma^\N$ be infinite words with $w=pu$ for some $p\in \Sigma^*$. Then $\mathrm{ACE}_\mathcal{I}(w) = \mathrm{ACE}_\mathcal{I}(u)$.
\end{theorem}
\begin{proof}
We may assume $\abs{p}=1$, since then the proof can be repeated to obtain the claim for arbitrary finite prefix. We can also assume that $\acei(u)$ is finite as we clearly have $\acei(u)\leq \acei(w)$.

Since $\mathrm{Fact}_n(u) \subseteq \mathrm{Fact}_n(w)$ we clearly have $\mathrm{ACE}_\mathcal{I}(u) \leq \mathrm{ACE}_\mathcal{I}(w)$ by Theorem  $\ref{thm: inside-h}$.

Assume then $\mathrm{ACE}_\mathcal{I}(u)+\lambda<\mathrm{ACE}_\mathcal{I}(w)$ for some positive $\lambda \in \R$. This means that there is $h \in \mathcal{I}$ such that there is arbitrary large factors $f$ of $w$ such that $\mathrm{E}(h(f))>\mathrm{ACE}_\mathcal{I}(u)+\frac{\lambda}{2}$. Since for factors of $u$ this cannot happen by definition, $f$ has to contain the letter $p$ at the start of $w$ and we have $f=pf'$.

Let $h(f)=x_f^{\mathrm{E}(h(f))}$. Choose $f$ so large that $\abs{h(f')} > \frac{ 3 \mathrm{ACE}_\mathcal{I}(u) \abs{ h(p)}}{\lambda }$ and $\frac{\abs{h(f')}}{\abs{x_f}}\leq \mathrm{ACE}_\mathcal{I}(u) + \delta$, where $\delta< \frac{\lambda}{9}(1+\frac{\lambda}{3 \mathrm{ACE}_\mathcal{I}(u)})^{-1}$. The later can be done just by choosing long enough $f'$ due to the definition of $\mathrm{ACE}_\mathcal{I}(u)$ and the fact that $h(f')$ is a part of repetition of $x_f$.

Now using obtained inequalities we have
\begin{align*}
  \mathrm{ACE}_\mathcal{I}(u)+\frac{\lambda}{2}&< \mathrm{E}(h(f)) = \frac{\abs{h(pf')}}{\abs{x_f}} \\
  &<\frac{\abs{h(f')}}{\abs{x_f}} +  \frac{\lambda \abs{h(f')}}{3 \mathrm{ACE}_\mathcal{I}(u)  \abs{x_f}} \\
  &\leq (\mathrm{ACE}_\mathcal{I}(u) + \delta)(1+\frac{\lambda }{3 \mathrm{ACE}_\mathcal{I}(u)  }) \\
  &= \mathrm{ACE}_\mathcal{I}(u) + \frac{\lambda}{3} + \delta (1+\frac{\lambda }{3 \mathrm{ACE}_\mathcal{I}(u)  }) \\
  &< \mathrm{ACE}_\mathcal{I}(u) + \frac{\lambda}{3} + \frac{\lambda}{9} = \mathrm{ACE}_\mathcal{I}(u) + \frac{4\lambda}{9}
\end{align*}

 a clear contradiction, so $\mathrm{ACE}_\mathcal{I}(u)=\mathrm{ACE}_\mathcal{I}(w)$.
\end{proof}

The above theorem does not hold for $\aceii$ by following example.

Let $u$ be an infinite word and let $a \not \in \alphabet(u)$. Now $au$ has factors $f$ of arbitrary length such that $\abs{f}_a = 1$. By Lemma $\ref{Lemma: single letter unbounded}$, those factors all have $\fei(f)=\infty$, and thus $\aceii(au)=\infty$. In the chapter about binary words, we see that $u$ can be chosen so that $\acei(au)=\acei(u)=1$.

More generally, $\acei$ and $\aceii$ relate to each other in a following way.
\begin{corollary}
\label{ACE< ACE'}
    For all words $w \in \Sigma^\N$, $\mathrm{ACE}_\mathcal{I} (w) \leq \mathrm{ACE}'_\mathcal{I}(w)$.
\end{corollary}
\begin{proof}
    Follows from Theorem $\ref{thm: inside-h}$ as
    \begin{align*}
    \mathrm{ACE}_\mathcal{I} (w)&=\sup_{h\in \inj} \big( \limsup_{n \to \infty} \{ \mathrm{E}(h(u)) \ \vert \ u \in \mathrm{Fact}_n(w) \} \big), \\
    & \leq \sup_{h\in \inj} \big( \limsup_{n \to \infty} \{ \mathrm{E}_\mathcal{I}(u) \ \vert \ u \in \mathrm{Fact}_n(w) \} \big), \\
    &= \limsup_{n \to \infty} \{ \mathrm{E}_\mathcal{I}(u) \ \vert \ u \in \mathrm{Fact}_n(w) \} = \mathrm{ACE}'_\mathcal{I}(w).
    \end{align*}

\end{proof}

\begin{theorem}
\label{thm:general-upperbound}
        For all words $w\in \Sigma^\N$, 
        \begin{displaymath}
            \acei(w) \leq (\abs{\Sigma}+1) +\abs{\Sigma}(\abs{\Sigma}+1)(\lfloor \ace(w) \rfloor+1 ).
        \end{displaymath}
\end{theorem}
\begin{proof}
Let $N = (\abs{\Sigma} + 1) + \abs{\Sigma} (\abs{\Sigma} + 1) (\lfloor \ace(w) \rfloor + 1)$
    and assume that $\acei(w) > N$ for some word $w \in \Sigma^\N$.
    By the definition of $\acei(w)$,
    there exists $h \in \mathcal{I}$ such that for all $n \in \N$,
    there exists $u_n^N \in \fact_+(h(w))$ with $\abs{u_n} > n$.

    By Theorem $\ref{lothaire 8.3.1}$, all long enough factors of $h(w)$ have $h(\Sigma)$ -degree of at most $\abs{\Sigma}$, or there exists arbitrary long factors $f=x^{\fe(f)}\in \fact_+(h(w))$ such that $\abs{x} \leq \max\{ \abs{h(a)} \mid a \in \Sigma \}$. From the latter, Lemma $\ref{lemma:unbounded-repetitions-of-x-in-image}$ gives us $\ace(w)=\infty$ and the proof ends. So assume that the former is the case.

    In the factor $u_n^N$, by the size of $N$,
    we can choose $\abs{\Sigma} + 1$ occurrences of $u_n$
    such that each occurrence pair chosen
    has a factor $u_n^{(\abs{\Sigma} + 1) (\lfloor \ace(w) \rfloor + 1)}$ in between them.
    In the chosen occurrences of $u_n$,
    at least two of the $h(\Sigma)$-interpretations
    we get from the natural $h(\Sigma)$-factorization of $h(w)$ are linked, because otherwise the $h(\Sigma)$-degree of $v$ would be at least $\abs{\Sigma}+1$, contradicting Theorem $\ref{lothaire 8.3.1}$.
    This means that in the factor $u_n^N$,
    there is a subfactor $h(x_n) = v^k$
    such that $x_n \in \fact_+(w)$,
    $v$ is a conjugate of $u_n$
    and $k \geq (\abs{\Sigma} + 1)(\lfloor \ace(w) \rfloor + 1)$.

    In the repetition $v^k$,
    consider the first $\abs{\Sigma} + 1$ occurrences of $v$.
    Again, we can conclude that two of the $h(\Sigma)$-interpretations of these must be linked as in the previous step.
    This means that we can write $x_n = r y^m t$ and $v = pq$
    so that $y$ is primitive,
    $h(r) \in (pq)^* p$, $h(y^m) \in (qp)^*$, $h(t) \in q(pq)^*$
    and $|h(y^m)| \leq |v^{\abs{\Sigma} + 1}|$.

    Now $h(rt), h(r y^m t) \in (pq)^*$, so $h(rt) h(r y^m t) = h(r y^m t) h(rt)$.
    By injectivity of $h$, $rtr y^m t = r y^m trt$ and thus $tr y^m = y^m tr$.
    It follows that $tr$ and $y^m$ have the same primitive root, namely, $y$.
    Because $x_n = r y^m t$ is a conjugate of $tr y^m$,
    we can write $x_n = z^j$, where $z$ is a conjugate of $y$.

    We have
    \begin{displaymath}
        \fe(x_n)
        = j
        = \frac{|h(z^j)|}{|h(z)|}
        = \frac{|v^k|}{|h(y)|}
        \geq \frac{|v^{(\abs{\Sigma} + 1)(\lfloor \ace(w) \rfloor + 1)}|}{|v^{\abs{\Sigma} + 1}|}
        = \lfloor \ace(w) \rfloor + 1
    \end{displaymath}
    and
    \begin{displaymath}
        \ace(w)
        \geq \limsup_{n \to \infty} \fe(x_n)
        \geq \lfloor \ace(w) \rfloor + 1,
    \end{displaymath}
    a contradiction.
    This means that the original claim must hold: $\mathrm{ACE}_{\mathcal{I}}(w) \leq N$.

 \end{proof}

We end this subsection by showing that, given large enough alphabet, $\acei$ can be large compared to $\ace$. For this, we need to construct words with low $\ace$. It follows from Beck's nonconstructive result \cite{beck1984}
that there are infinite binary words with $\ace$ equal to $1$.
Cassaigne in \cite{cassaigne2008extremal} gives an explicit construction of such a word.

\begin{theorem}[\cite{beck1984} or \cite{cassaigne2008extremal}]
\label{thm:ace=1}
    There exists an infinite binary word $u$ with $\ace(u)=1$.
\end{theorem}
\begin{theorem}
\label{thm:big-acei}
    For all $n$, there exists a $2n$-ary infinite word $w_n$ with $\ace(w_n)=1$ such that $\acei(w_n) \geq n.$

\end{theorem}
\begin{proof}
    Let $u$ be a binary word with $\ace(u)=1$ as in Theorem $\ref{thm:ace=1}$. For $n=1$, this word already proves the statement, so let $n\geq 2$.
    Let $a_1, \dots, a_n, b_1, \dots b_n$ be distinct letters. Let $u_i \in \{a_i, b_i \}^\N$ be a version of $u$ defined over $\{a_i, b_i \}$ for all $i \leq n$ (that is, there is a renaming of letters such that $u_i$ turns into $u$). Let $u_i = u_{i,1}u_{i,2}u_{i,3} \ldots$ where $\abs{u_{i,j}}=j$. Now let
    \begin{displaymath}
        w_n = \prod_{j=1}^{\infty} u_{1,j}u_{2,j}u_{3,j} \ldots u_{n,j},
    \end{displaymath}
    which is $2n$-ary word.

    Let $(x_l)$ be a sequence of factors of $w_l$ such that $\lim_{l \to \infty} \abs{x_l} = \infty$ and $\lim_{l \to \infty} \fe(x_n) = \ace(w_n)$. If for all large $l$ the word $x_l$ is a factor of some $u_{i_l}$ for some $i_l$, then $\ace(w_n)\leq \max \ace(u_i)=1$. Also, if for all large $l$ the word $x_l$ has $\fe(x_l)=1$, then $\ace(w_n)=1$. So assume that $x_l$ has alphabets from two different sets $\{a_i, b_i \}$ and $\{a_j, b_j \}$ and $\fe(x_l)>1$ for all large $l$. Let $x_{l,i}$ be the largest subsequence of $x_l$ such that $x_{l,i} \in \{a_i, b_i\}^*$, which is then a factor of $u_i$.

     By $\fe(x_l)>1$, some binary alphabet $\{a_j, b_j\}$ appears in $x_l$ twice with all the other alphabets in between these appearances. Then we have $\lim_{l \to \infty} \abs{x_{l,i}} = \infty$ for all $i\neq j$. If $\lim_{l \to \infty} \abs{x_{l,j}} \neq \infty$, it follows that $\lim_{l \to \infty} \fe(x_l)=1$, and thus $\ace(w_n)=1$. So we assume that $x_{l,i}$ gets arbitrarily large for all $i\leq n$ as $l$ increases.

    Let $x_l = y_l^{\fe(x_l)}$ for some word $y_l$ and let $y_{l,i}$ be the largest subsequence of $y_l$ with $y_{l,i} \in \{a_i,b_i\}^*$. Then from $x_l = y_l^{\fe(x_l)}$ it follows that $x_{l,i}=y_{l,i}^{r_{l,i}}$ for some rational $r_{l,i}$. Now
    \begin{align*}
       \ace(w_n)=&\lim_{l \to \infty} \fe(x_n)= \lim_{l \to \infty} \frac{\abs{x_l}}{\abs{y_l}}
       = \lim_{l \to \infty}  \frac{ \sum_{i=1}^{n} \abs{x_{l,i}}}{\abs{y_l}} \\
       =& \lim_{l \to \infty}\frac{\sum_{i=1}^{n} \frac{\abs{x_{l,i}}}{\abs{y_{l,i}}}\abs{y_{l,i}}}{\abs{y_l}} \leq \lim_{l \to \infty} \max_{i \leq n} \bigg\{ \frac{\abs{x_{l,i}}}{\abs{y_{l,i}}} \bigg\} \frac{\sum_{i=1}^{n} \abs{y_{l,i}}}{\abs{y_l}} = 1,
    \end{align*}
so $\ace(w_n)=1$.

   Let $h \in \inj$ be the morphism defined by
    \begin{displaymath}
        h(a_i) = c^{i - 1} a c^{n - i},\qquad
        h(b_i) = c^{i - 1} b c^{n - i}
    \end{displaymath}
    for all $i \leq n$.
    We have by construction of $h$ (similarly to Theorem $\ref{thm: highpower}$)
    \begin{align*}
        h(u_{1,j}u_{2,j}u_{3,j} \ldots u_{n,j}) =(h(u_{1,j})c)^{n-\frac{1}{j}},
    \end{align*}
    So $\acei(w_n)\geq n$.

 \end{proof}

Unfortunately, the upper bound given in $\ref{thm:general-upperbound}$ is quadratic in the size of the alphabet, while the example above gives linear lower bound. In the next two section, we will tighten this gap in the special case of binary alphabet.

\subsection{Synchronizing words}

In \cite{ochem2024critical}, Dvo\v{r}\'{a}kov\'{a}, Ochem and Opo{\v{c}}ensk{\'a} studied $\ace$ of words with a few distinct palindromes. As a tool, they showed sufficient conditions for a morphism $h$ and a word $u$ to satisfy $\ace(h(u))=\ace(u)$, stated in Theorem $\ref{thm:synchronizing-case-old}$. In Theorem $\ref{thm:synchronizing-case}$, we drop one of the conditions on the word $u$ and weaken the conclusion, making the statement interesting in studying $\acei$. We also take binary words as a special case, as we consider the binary case again in a later subsection of this paper.

We use slightly adjusted notation from \cite{ochem2024critical}. Let $X$ be a finite set of finite words. We say that a word $w$ is a \emph{synchronizing word (of $X$)} if there exist words $w_1, w_2$ such that $w=w_1w_2$ and for all words $v\in X^*$, if $v=pws$ for some words $p$ and $s$, then $v=v_1v_2$ and $v_1 = pw_1 \in X^*$ and $v_2=w_2s \in X^*$ for some words $v_1$ and $v_2$. We also then say that $w_1$ is a synchronizing prefix and $w_2$ is a synchronizing suffix.
Note that if $w$ has a synchronizing word as a factor, then $w$ is also synchronizing.
While we do not need any old results on synchronizing words, we want to note that synchronizing words have appeared in different forms for a long time in coding theory, earliest explicit mention we found being from 1965 by Sch{\"u}tzenberger \cite{schutzenberger1965codes}.

Let $h: \Sigma^* \to \Gamma^*$ be a morphism. If a word $w$ is a synchronizing word of $h(\Sigma)$, then we also say that it is a synchronizing word of $h$. The following Lemma demonstrates one way how synchronizing words are useful.
\begin{lemma}
\label{lem: sync-word-help}
    Let $h \in \inj$, $t\in \Sigma^*$, $f\in \fact(h(t))$ and $w=w_1w_2, w'=w'_1w'_2$ synchronizing words of $h$ with synchronizing prefixes $w_1$ and $w'_1$. Now:
    \begin{enumerate}
        \item If $f=pwmw's$ for some words $p,m,s \in \Gamma^*$, then $f=uu'$ for some synchronizing words $u$ and $u'$ of $h$.
        \item If $f=ww'$, then $h^{-1}(w_2w'_1)\in \fact(t)$.
    \end{enumerate}
\end{lemma}
\begin{proof}
    Claim $1.$ is clear from the definition of synchronizing words. For $2.$, let $h(t)=pw_1w_2w'_1w'_2s$ for some words $p,s\in \Gamma^*$. By definition of synchronizing words, we have some words $t_1,t_2,t_3,t_4$ such that $h(t_1)=pw_1$, $h(t_2)=w_2w'_1w'_2s$, $h(t_3)=pw_1w_2w'_1$ and $h(t_4)=w'_2s$.
    Now $h(t)=h(t_1t_2)=h(t_3t_4)$, and by injectivity, $t=t_1t_2=t_3t_4$. Clearly $\abs{t_3}\geq \abs{t_1}$, so $t_3=t_1v$ for some $v\in \Sigma^*$ with $h(v)=w_2w'_1$.
\end{proof}
Here and later in Theorem $\ref{thm: uniform-binary}$ infinite word $w \in \Sigma^\N$ is said to have \emph{uniform letter frequencies} if for every letter $a \in \text{alph}(w)$ we have
\begin{equation*}
    \liminf_{\substack{f \in \fact(w) \\ \abs{f} \to \infty}} \frac{\abs{f}_a}{\abs{f}} = \limsup_{\substack{f \in \fact(w) \\ \abs{f} \to \infty}} \frac{\abs{f}_a}{\abs{f}}
\end{equation*}

\begin{theorem}[Dvo\v{r}\'{a}kov\'{a}, Ochem and Opo{\v{c}}ensk{\'a} \cite{ochem2024critical}]
\label{thm:synchronizing-case-old}
    Let $h$ be an injective morphism, $u$ an infinite word with uniform letter frequencies and $L$ a natural number. If all factors $f$ of $h(u)$ with $\abs{f}\geq L$ are synchronizing words of $h$, then $\ace(h(u))=\ace(u)$.
\end{theorem}

We have used the proof of $\ref{thm:synchronizing-case-old}$ directly from \cite{ochem2024critical} in the proof of Theorem $\ref{thm:synchronizing-case}$ \footnote{With a personal permission from the authors of the original proof.}. To emphasize that some of the work in the proof is done by Dvo\v{r}\'{a}kov\'{a}, Ochem, and Opo{\v{c}}ensk{\'a}, we use $\textbf{bold text}$ on the parts where the proof is the same as in their text.

\begin{theorem}
\label{thm:synchronizing-case}
    Let $h$ be an injective morphism, $u\in \Sigma^\N$ an infinite word and $L$ a natural number. If all factors $f$ of $h(u)$ with $\abs{f}\geq L$ are synchronizing words of $h$ and $\ace(u)\in \R$, then $\ace(h(u))\leq \lceil \ace(u) \rceil$. Moreover, if $u$ is over a binary alphabet and $\ace(u) \not \in \N$, then there exists $\lambda_u>0$ such that for all $h$ and $L$ pairs, $\ace(h(u)) \leq \lceil \ace(u) \rceil - \lambda_u$.
\end{theorem}
\begin{proof}
    \textbf{\boldmath
    Let  $h \in \mathcal{I}$ such that $\ace(h(u)) \geq \ace(u)$. Let $(w_n)$ and $(v_n)$ be sequences of words such that 
    }
    \boldmath
    \begin{itemize}
        \item
        $\lim_{n \to \infty} \abs{v_n}=\infty$,
        \item
        $w_n \in \fact_+(h(u))$ \textbf{for all} $n$,
        \item
        $w_n$ \textbf{is a prefix of the periodic word} $v_n^\omega$ \textbf{for all} $n$ \textbf{and}
        \item
        $\lim_{n\to \infty} \frac{\abs{w_n}}{\abs{v_n}}  = \ace(h(u)).$
    \end{itemize}
    
   \textbf{If $\ace(h(u))=1$, then $\ace(u)=1$ and the claim holds. So assume that $\ace(h(u))>1$. Then for large enough $n$, $\abs{w_n}>\abs{v_n}$. Since $v_n$ will be arbitrarily large, we have by the assumption that for large enough $n$, $v_n$ and $w_n$ will have a non-overlapping prefix and a suffix each being a synchronizing words. By Lemma $\ref{lem: sync-word-help}$ we then have unique words $w_n'$ and $v_n'$ such that:
   \begin{displaymath}
       w_n= x_n h(w_n') y_n \textbf{ and } v_n = x_nh(v_n') z_n,
   \end{displaymath}
    where $x_n, y_n$ and $z_n$ are minimally chosen so that $x_n$ contains (is) a synchronizing prefix and $y_n$ and $z_n$ contain (are) synchronizing suffixes of synchronizing words in $w_n$ and $v_n$. By Lemma \ref{lem: sync-word-help}, it follows that $w_n'$ and $v_n'$ are both factors of $u$. We also have that $\abs{x_n}<L, \abs{y_n}<L$ and $\abs{z_n}<L$.}

    \textbf{
    Now since $w_n$ is a repetition of $v_n$, we may write $w_n = v_n^{k_n} v_n^{r_n}$ for some $k_n\in \N$ and $r_n<1$. The proof splits now into two cases depending on if $\abs{v_n^{r_n}}$ is bounded or not:} \\
    \textbf{
    CASE 1. ($\abs{v_n^{r_n}}$ is bounded):
    In this case, since $\ace(h(u))>1$, we will have that $k_n \geq 2$ for large enough $n$.
    Then since
    \begin{displaymath}
        w_n = v_n^{k_n} v_n^{r_n} = (x_nh(v_n') z_n)^2 (x_nh(v_n') z_n)^{k_n-2} v_n^{r_n},
    \end{displaymath}
    by Lemma $\ref{lem: sync-word-help}$, there exists an unique factor $t_n$ of $u$ such that $h(t_n)=h(v_n')z_n x_n$ and $v_n'$ is a prefix of $t_n$. Thus $w_n = x_n h(t_n^{k_n-1} v_n')z_n v_n^{r_n}$. Now
    \begin{align*}
        \ace(h(u)) &= \lim_{n\to \infty} \frac{\abs{w_n}}{\abs{v_n}} = \lim_{n\to \infty} \frac{\abs{x_n h(t_n^{k_n-1} v_n')z_n v_n^{r_n}}}{\abs{v_n}} = \lim_{n\to \infty} \frac{\abs{h(t_n^{k_n-1} v_n')}}{\abs{h(t_n)}} \\
        &= k_n-1+ \lim_{n\to \infty} \frac{\abs{h(v_n')}}{\abs{h(t_n)}},
    \end{align*}
    where the second to last equality follows form the fact that $\abs{x_n}$, $\abs{z_n}$ and $\abs{v_n^{r_n}}$ are all bounded and $\abs{h(t_n)}= \abs{v_n}$.}
    \unboldmath

    Since $v_n'$ is a prefix of $t_n$, we have that $\lim_{n\to \infty} \frac{\abs{h(v_n')}}{\abs{h(t_n)}} \leq 1$. Now, since $\lim_{n\to \infty} \abs{t_n} = \infty$ and $t_n^{k_n-1} v_n'$ is a factor of $u$, we must have that $k_n-1 \leq \lfloor \ace(u) \rfloor$ for large enough $n$. If $\ace(u) \not \in \N$, then we have that $\ace(h(u)) \leq \lfloor \ace(u) \rfloor + 1 = \lceil \ace(u) \rceil$. If $\ace(u) \in \N$, then $k_n-1$ can at most equal to $\ace(u)$ for large $n$, but then we must have that $\lim_{n \to \infty}\frac{\abs{v_n'}}{\abs{t_n}} = 0$. For any morphism, this leads to $\lim_{n \to \infty}\frac{\abs{h(v_n')}}{\abs{h(t_n)}}=0$, so in that case we have $\ace(h(u)) \leq \lfloor \ace(u) \rfloor=\acei(u)$.
    \\
    \boldmath
    \textbf{
    CASE 2. ($\abs{v_n^{r_n}}$ is unbounded): In this case, by choosing a subsequence of $(w_n)$ and $(v_n)$ if necessary, we may assume that for large enough $n$ we have a factorization $v_n^{r_n} = x_n h(v_n'')y_n$, where (again) $x_n$ is a minimal word containing a synchronizing prefix of a synchronizing word that is a factor of $v_n^{r_n}$, and $y_n$ is the minimal word containing a synchronizing suffix of a synchronizing word that is a factor of $v_n^{r_n}.$ It follows that $v_n''$ is a factor of $u$. Now $w_n = (x_n h(v_n')z_n)^{k_n} x_n h(v_n'')y_n$. Again, by Lemma $\ref{lem: sync-word-help}$, there exists a unique factor $t_n$ of $u$ such that $h(t_n)= h(v_n')z_n x_n$ and $t_n^{k_n} v_n''$ is a factor of $u$. Now, very similarly to the case 1.
    \begin{align*}
        \ace(h(u)) &= \lim_{n\to \infty} \frac{\abs{w_n}}{\abs{v_n}} = \lim_{n\to \infty} \frac{\abs{x_n h(t_n^{k_n} )h(v_n'')y_n}}{\abs{v_n}} = \lim_{n\to \infty} \frac{\abs{h(t_n^{k_n} v_n'')}}{\abs{h(t_n)}} \\
        &= k_n+ \lim_{n\to \infty} \frac{\abs{h(v_n'')}}{\abs{h(t_n)}},
    \end{align*} where the second to last equality holds since $\abs{h(t_n)}=\abs{v_n}$ and $\abs{x_n}$ and $\abs{y_n}$ are bounded.}  
    \unboldmath Now once again, since $t_n^{k_n} v_n''$ is a factor of $u$, $k_n\leq \lfloor \ace(u) \rfloor$ for large enough $n$. Also since $v_n''$ is a prefix of $t_n$, $\lim_{n\to \infty} \frac{\abs{h(v_n'')}}{\abs{h(t_n)}}\leq 1$. If $\ace(u)\not \in \N$, the claim $\ace(h(u))\leq \lceil \ace(u) \rceil$ follows. If $\ace(u) \in \N$, then $k_n$ can be at most equal to $\ace(u)$ for large $n$, but then we must have that $\lim_{n \to \infty}\frac{\abs{v_n''}}{\abs{t_n}} = 0$. For any morphism, this leads to $\lim_{n \to \infty}\frac{\abs{h(v_n')}}{\abs{h(t_n)}}=0$, finishing the proof of non-binary case as in both cases we have $\ace(h(u))\leq \lceil \ace(u) \rceil$.

    What is left is the special case when $u \in \{a,b\}^\N$. In both cases 1.\ and 2., if $\lfloor \ace(u) \rfloor <\ace(h(u))$, then there was a sequence of factors $t_n^{l_n}p_n$ of $u$ (in the case 1., $l_n=k_n-1$ and $p_n=v_n'$ and in the case 2. $l_n = k_n$ and $p_n=v_n''$) such that $p_n$ is a prefix of $t_n$, $\lim_{n \to \infty} \abs{t_n} = \infty$, $\lim_{n\to \infty} l_n = \lfloor \ace(u) \rfloor$ and $\lim_{n \to \infty}\frac{\abs{p_n}}{\abs{t_n}} > 0$, $\lim_{n \to \infty} \abs{p_n} = \infty$. Let $F$ be the set of such factor sequences. Then, by the cases 1.\ and 2., for each morphism $h$ that follows this theorem's assumptions and has $\lfloor \ace(u) \rfloor <\ace(h(u))$, we have some $(t_n^{l_n}p_n) \in F$ such that $\ace(h(u))= \lfloor \ace(u) \rfloor + \lim_{n \to \infty} \frac{\abs{h(p_n)}}{\abs{h(t_n)}}$. We will show that there exists $\lambda_u>0$ such that for all $h' \in \mathcal{I}$ and $(t_n^{l_n}p_n) \in F$, $\lim_{n \to \infty} \frac{\abs{h'(p_n)}}{\abs{h'(t_n)}} \leq 1-\lambda_u$.

    So let $(t_n^{l_n}p_n) \in F$ and $h' \in \inj$. If $\limsup_{n \to \infty} \frac{\abs{p_n}}{\abs{t_n}}=1$, then $\limsup_{n \to \infty} \fe(t_n^{l_n}p_n)>\ace(u)$, a contradiction. This means that if $t_n = p_n s_n$, then $\liminf_{n\to \infty} \frac{\abs{s_n}}{\abs{t_n}} \geq r>0$.
    This also means that if the quantity
    $\limsup_{n \to \infty} \frac{\abs{s_n}_a}{\abs{s_n}}=0$ or $\limsup_{n \to \infty} \frac{\abs{s_n}_b}{\abs{s_n}}=0$, then $u$ has arbitrarily long repetition of $b$ or $a$ and thus $\ace(u)=\ace(h'(u))=\infty$. So there is $\delta_a \in (0,1)$ (resp. $\delta_b \in (0,1)$) such that $ \liminf_{n \to \infty} \frac{\abs{s_n}_a}{\abs{s_n}} \geq \delta_a$. We can choose the lower bounds $r, \delta_a$ and $\delta_b$ so that they are the same for all sequences in $F$. If no such $r$, $\delta_a$ or $\delta_b$ exists, then we could construct a sequence belonging to $F$ such that at least one of the above quantities is zero by picking elements from sequences that lead to smaller and smaller lower bound.

    Now, just by the properties of positive real numbers, 
    \begin{displaymath}
        \frac{\abs{h'(p_n)}}{\abs{h'(t_n)}} = \frac{\abs{p_n}_a\abs{h'(a)} + \abs{p_n}_b\abs{h'(b)}}{\abs{t_n}_a\abs{h'(a)} + \abs{t_n}_b\abs{h'(b)}} \leq \max \left\{ \frac{\abs{p_n}_a\abs{h'(a)}}{\abs{t_n}_a\abs{h'(a)}}, \frac{\abs{p_n}_b\abs{h'(b)}}{\abs{t_n}_b\abs{h'(b)}} \right\}.
    \end{displaymath}
    For the quantities inside the maximum function we have 
    \begin{displaymath}
        \frac{\abs{p_n}_a}{\abs{t_n}_a} = \frac{\frac{\abs{p_n}_a}{\abs{t_n}}}{\frac{\abs{p_n}_a + \abs{s_n}_a }{\abs{t_n}}} \leq \frac{1}{1+\frac{\abs{s_n}_a}{\abs{t_n}}}=\frac{1}{1+\frac{\abs{s_n}_a}{\abs{s_n}}\frac{\abs{s_n}}{\abs{t_n}}}\leq \frac{1}{1+\delta_ar}=1-\frac{\delta_ar}{1+\delta_ar}
    \end{displaymath}
    and by the same calculation $\frac{\abs{p_n}_a}{\abs{t_n}_a} \leq 1-\frac{\delta_b r}{1+\delta_b r}$.
    Combining these we get 
    \begin{displaymath}
        \frac{\abs{h'(p_n)}}{\abs{h'(t_n)}} \leq 1 - \min \left\{ \frac{\delta_a r}{1+\delta_a r}, \frac{\delta_b r}{1+\delta_b r} \right\},
    \end{displaymath}
    and by above reasoning this upper bound does not depend on the choice of $h'$ or the sequence $(t_n^{l_n}p_n)$. Thus we can choose $\lambda_u = \min \left\{ \frac{\delta_a r}{1+\delta_a r}, \frac{\delta_b r}{1+\delta_b r} \right\}$ and the proof is done.
 \end{proof}

\subsection{Binary case}

In previous subsection, we saw that if in the word $u$ all large enough factors are synchronizing words of $h \in \inj$, then $\ace(h(u))$ is not too far from $\ace(u)$. For binary words, we will show that the complement case leads to arbitrary long repetitions. (Theorem $\ref{thm: non-sync-case}$). This means that for binary words, one cannot increase $\ace$ too much with injective morphism. However, an increase less than one is possible by Theorem $\ref{thm: optimal-binary-example}$. We start by two lemmas, the first of which is the reason we needed to introduce shifted sequences in preliminaries.

\begin{lemma}
\label{lemma:bi-infinite2factorizations}
    Let $X$ be a finite set of words and
    let $w_i$ be words such that $\lim_{i \to \infty} |w_i| = \infty$.
    If every $w_i$ has an $X$-interpretation but is not a synchronizing word of $X$,
    then we can write $w_i = u_i v_i$ so that a subsequence of $(u_i \sep v_i)_{i \in \N}$
    converges to a bi-infinite word with two different $X$-factorizations.
\end{lemma}

\begin{proof}
    Let $w_i$ have an $X$-interpretation $(x_{i, 0}, \dots, x_{i, m_i})$.
    Let $k_i = \lfloor m_i / 2 \rfloor$ and
    \begin{displaymath}
        u_i = x_{i, 0} \dotsm x_{i, k_i - 1}, \qquad
        v_i = x_{i, k_i} \dotsm x_{i, m_i}.
    \end{displaymath}
    Because $u_i$ is not a synchronizing prefix of $w_i$,
    $w_i$ must have an $X$-interpretation $(y_{i, 0}, \dots, y_{i, n_i})$
    such that for some $l_i$ we have $y_{i, l_i} = p_i q_i$, $p_i \ne \eps \ne q_i$ and
    \begin{displaymath}
        u_i = y_{i, 0} \dotsm y_{i, l_i - 1} p_i, \qquad
        v_i = q_i y_{i, l_i + 1} \dotsm y_{i, n_i}.
    \end{displaymath}
    By Lemma~\ref{lem:compact-bi},
    there exists a function $g$ such that the $g$-subsequences of the sequences
    \begin{displaymath}
        ((x_{i, 0}, \dots, x_{i, k_i - 1}, \eps \sep x_{i, k_i}, \dots, x_{i, m_i}))_{i \in \N},
        \qquad
        ((y_{i, 0}, \dots, y_{l_i - 1}, p_i \sep q_i, y_{l_i + 1}, \dots, y_{i, n_i}))_{i \in \N}
    \end{displaymath}
    both have a limit.
    Let these limits be $S$ and $T$, respectively.
    It can be verified that both $S$ and $T$ are $X$-factorizations of the bi-infinite word
    that is the limit of the $g$-subsequence of $(u_i \sep v_i)_{i \in \N}$.
    Moreover, $S(-1) = \eps \ne T(-1)$, so $S$ and $T$ are different.
 \end{proof}

\begin{lemma}
    \label{thm: non-sync-case}
    Let $h$ be an injective morphism and $w$ infinite binary word. If there exists infinite list $f_1, f_2, f_3, \ldots$ of factors of $h(w)$ that are not synchronizing words of $h$ and $\lim_{i \to \infty} \abs{f_i} = \infty$, then $\ace(h(w))=\ace(w)=\infty$.
\end{lemma}
\begin{proof}
    By Lemma $\ref{lemma:bi-infinite2factorizations}$ there exists a bi-infinite word $u$ that has at least two different $\{h(a),h(b)\}$-factorizations such that $\fact_+(u) \subseteq \fact_+(h(w))$. By Corollary $\ref{lemma:bi-infinite-periodic-word}$ we see that $u$ is then periodic, $u= \dotsm xxx \dotsm$ for some $x\in \Gamma^*$. Thus $h(w)$ has a factor $x^l$ for each $l\in \N$. By Lemma $\ref{lemma:unbounded-repetitions-of-x-in-image}$, we see that $\ace(w)=\acei(w)=\infty.$

\end{proof}

\begin{theorem}
    \label{thm: uniform-binary}
    Let $w$ be an infinite binary word with $\ace(w)\in \R$. If uniform letter frequencies exist in $w$ or $\ace(w)\in \N$, then $\acei(w)=\ace(w)$. Otherwise, $\acei(w) < \lceil \ace(w) \rceil$.
\end{theorem}
\begin{proof}
    Let $h\in \mathcal{I}$. If $h(w)$ has arbitrarily large factors that are not synchronizing, then Lemma $\ref{thm: non-sync-case}$ states the claim as $\acei(w)=\ace(w)=\infty$. (Or more accurately we have a contradiction with assumption that $\ace(w)\in \R$.)

    Assume then that all large enough factors of $h(w)$ are synchronizing words. If $w$ has uniform letter frequencies, then by Theorem $\ref{thm:synchronizing-case-old}$ we have $\ace(h(w))=\ace(w)$ and the claim follows as $h\in \mathcal{I}$ was arbitrary. Otherwise Theorem $\ref{thm:synchronizing-case}$ states that $\ace(h(w)) \leq \lceil \ace(w) \rceil$. If then $\ace(w)\in \N$, it is clear that $\ace(h(w))=\ace(w)$, leading to $\acei(w)=\ace(w)$ as $h \in \inj$ was arbitrary. Finally if $\ace(w) \not \in \N$, then by $\ref{thm:synchronizing-case}$ we have some $\lambda_w>0$ depending only from $w$, such that $\ace(h(w))<\lceil \ace(w) \rceil - \lambda_w$, and again as $h\in \mathcal{I}$ was arbitrary, the claim follows.
\end{proof}

\begin{corollary}
    There exists an infinite binary word $w$ with $\acei(w)=1$.
\end{corollary}
\begin{proof}
    Theorem $\ref{thm:ace=1}$ and Corollary $\ref{thm: uniform-binary}$.
\end{proof}

The rest of this subsection is dedicated to showing that the bound given in Corollary $\ref{thm: uniform-binary}$ is tight.

\begin{lemma}
\label{lemma:acefixingmorphism}
    Let $A=\{c_1, c_2, c_3 \ldots c_d\}$, $f:\{1,2,3,\ldots,d\} \to \N$ injective and $h$ a morphism defined by 
    \begin{displaymath}
      h(c_i)=a^{m-f(i)}b^{f(i)}  
    \end{displaymath}
    for all $c_i \in A$, where $m \geq \max\{f(i)\}+1$. Then for all $u \in A^\N$, $\ace(h(u))=\ace(u).$
\end{lemma}
\begin{proof}
    In \cite{cassaigne2008extremal}, Cassaigne proved this lemma in the special case $f(i)=i$ and $m=d+1$. Take $A'=A \cup \{c_{d+1}, c_{d+2}, c_{d+3}, \ldots ,c_{m-1}\}$, construct $h'$ for that alphabet with $f'(i)=i$ and $m'=m$ such that $h'$ preserves $\ace$. Then define $A'' \subseteq A'$ as $A''=\{c_{f(1)}, c_{f(2)}, \ldots, c_{f(d)}\}$. Then $h= h'_{\vert A''}$ preserves $\ace$ for all $u \in A''^\N$, which is equivalent to the claim after renaming the letters back to original.
\end{proof}

\begin{theorem}
\label{thm: optimal-binary-example}
    Let $n\in \N$, $n\geq1$, $\lambda \in (0,1)$ and $\delta>0$. There exists an infinite binary word $w_{(n,\lambda, \delta)}$ such that \[ n<\ace(w_{(n,\lambda, \delta)})\leq n+\lambda \] and \[ \acei(w_{(n,\lambda, \delta)})\geq n+1-\delta.\]
\end{theorem}
In the proof we have several sequences dependent on parameter $l\in \N$ (sometimes indirectly as a factors of some words other sequence). We use following terminology when talking about these sequences:

Let $x=x_l$ and $y=y_l \in \fact_+(x)$.
If $\limsup_{l\to\infty} \frac{\abs{y}}{\abs{x}}>0$ then we say that $y$ is a \emph{positive portion} of a word $x$.

We say that \emph{most of $x$ is in $y$} if $\liminf_{l \to \infty}\frac{\abs{y}}{\abs{x}}=1$. 

If $\limsup_{l \to \infty}\frac{\abs{y}}{\abs{x}}<1$, then we say that $y$ is a \emph{bounded portion} of $x$. 

Finally, if for any words $u=u_l$ and $v=v_l$ we have $\liminf_{l \to \infty}\frac{\abs{u}}{\abs{v}}>0$, then $u$ is of length \emph{proportional to} $v$.

\begin{proof}
    Let $A=\{c_1,c_2\}$, $B=\{c_3,c_4\}$ and $C=\{c_5, c_6\}$ be disjoint alphabets and $\phi$ be a morphism defined by $\phi(c_1)=c_3$ and $\phi(c_2)=c_4$. Let $k>\frac{1}{\lambda}$ be a fixed integer. Let \[ u = u_1u_2 u_3 \ldots
     \] and \[
     \phi(u)= v_1 v_2 v_3 \ldots\]
     where $\abs{v_i}+1=k (\abs{u_i}+1)$ and $\lim_{i \to \infty} \frac{\abs{u_i}(k+1)i}{\abs{u_{i+1}}}=0$ and $\ace(u)=\ace(\phi(u))=1$.

    Now let \[w=\prod_{i=0}^\infty (u_{i} c_5 v_i c_5 )^n u_i c_5 c_6 .\]

    Let $h(c_1)=ab^{m-1},$ $h(c_2)=aab^{m-2}$, $h(c_3)=a^{m-2}bb$, $h(c_4)=a^{m-1}b$, $h(c_5)=a^{m-3}bbb$ and $h(c_6)=a^{m-4}bbbb$ where $m> \frac{2+2k}{\delta} +2k+6$. Then $h(w)$ is a suitable choice for $w_{(n, \lambda, \delta)}$ by following two claims.
    \bigskip

    \noindent CLAIM 1 ($\ace(h(w)) =  n+\frac{1}{k+1}$):

    By Lemma $\ref{lemma:acefixingmorphism}$, $\ace(h(w))=\ace(w)$. So we prove this claim for the word $w$. Let $(x_l)$ be the sequence of factors of $w$ such that $x_n^{r_l}$ is a factor of $w$, $\lim_{l \to \infty}\abs{x_l}=\infty$ and $\lim_{l\to \infty}r_l=\ace(w)$. Clearly if we choose $x_l=u_l c_5 v_l c_5$ we get $r_l =n +  \frac{1}{k+1}$, so we may assume that $r_l \geq n + \frac{1}{k+1}$.

     Assume that we have infinite number of different choices for $l$ such that $\abs{x_l}_{c_6} \geq 1$. For those $l$, we have a factorization $x_l=t_1 c_6 t_2$ for some words $t_1$ and $t_2$ depending on $l$. This leads to two cases and several subcases, all of which will lead to a contradiction.
    \bigskip

    \noindent CASE 1 ($t_2= u_i c_5 t'_2$ for some $i$ and word $t'_2$.):

    \begin{minipage}{.8\textwidth}
    Since $u_i$ is much larger than $t_1$ by construction, $x_l^{r_l}$ starts with $t_1 c_6 u_i c_5 t'_2 t_1 c_6 u'_i$ for some prefix $u'_i$ of $u_i$ in order to have exponent $r_l$. The latter $c_6$ has to fit some $c_6$ in $w$, so either $t'_2$ is most of the word $x_l$, or $u'_i$ is a prefix of $u_{i+1}$.

    SUBCASE 1.1 ($t'_2$ is most of the word $x_l$):

    In this case $x_l^{r_l}$ starts with $t_1 c_6 u_i c_5 t'_2 t_1 c_6 u_i c_5$, leading to contradiction as there is only one occurrence of factor $c_6 u_i c_5$ in $w$.

    SUBCASE 1.2 ($u'_i$ is a prefix of $u_{i+1}$):

    In this case, $u'_i$ is a positive portion of $u_i$. Now in $u$ we have arbitrary large factors of the form $u'_i u''_i u'_i$ where $u'_iu''_i=u_i$ and $\fe(u'_i u''_i u'_i)>1+s$ for some fixed $s\in \R$. This is a contradiction with the fact that $\ace(u) = 1$.
    \end{minipage}

    \bigskip

    \noindent CASE 2 ($t_2= u'_i$ for some prefix $u'_i$ of $u_i$ for some $i$.):

    \begin{minipage}{.8\textwidth}
    SUBCASE 2.1 (most of the $x_l$ is in $u'_i$):

    In this case $x_l^{r_l}$ starts with $t_1 c_6 u'_i t_1 c_6$. By construction, these occurrences of $c_6$ are too close of each other to be part of $w$.

    SUBCASE 2.2 ($u'_i$ is a bounded portion of $x_l$):

    In this case $x_l^{r_l}$ starts with $t_1 c_6 u'_i t'_1$ for some prefix $t'_1$ of $t_1$. If $\abs{t_1}_{c_6} \geq 1$, then by construction $\abs{t'_1}_{c_6}\geq 1$, and this leads $x_l^{r_l}$ having a factor not in $w$. Thus $t'_1$ starts with a prefix of $u_{i-1}$. But now $u_i$ has a factor $u'_{i-1}$ in a prefix of length proportional to $u'_{i-1}$, which is again a contradiction with $\ace(u)=1$ similar to the SUBCASE 1.2.
    \end{minipage}

    \bigskip

    Thus both CASE 1 and CASE 2 lead to a contradiction, and we must have $\abs{x_l}_{c_6}=0$ for all large enough $l$. Thus for all large enough $l$ there is some $i$ such that $x_l^{r_l}$ is fully contained in factor $(u_{i} c_5 v_i c_5 )^n u_i c_5$. Since $r_l \geq n+\frac{1}{k+1}$, we see that $\abs{x_l}\leq \abs{u_i c_5 v_i c_5}$. This leads to $\abs{x_l}_{c_5} \leq 2$, which we split again to three cases.

    \bigskip

    CASE A ($\abs{x_l}_{c_5}=0$):
    In this case, by construction, $x_l$ is a factor of $u$ or $\phi(u)$, which would lead to $\ace(w)=1$, a contradiction since $r_l\geq n+\frac{1}{k+1}>1$.

    \bigskip

    CASE B ($\abs{x_l}_{c_5}=1$):
    In this case $x_l = t_1 c_5 t_2$ for some words $t_1$ and $t_2$ that are from different alphabets. Then $x_l^{r_l}$ starts with $t_1 c_5 t_2 t'_1$, leading to the fact that $t_1$ is empty, as otherwise the alphabets of each side of $c_5$ do not change correctly. Thus $x_l^{r_l}$ starts with $c_5 t_2 c_5 t'_2$, again leading to $c_5$ not changing alphabets as seen in the construction of $w$.

    \bigskip

    CASE C ($\abs{x_l}_{c_5}=2$):
    In this case $x_l = v'_i c_5 u_i c_5 v''_i$ or $x_l = u'_i c_5 v_i c_5 u''_i$. If $x_l^{r_l}$ has at least $3$ occurrences of $c_5$, then to fit it with the factors of $w$, we see that $\abs{x_l}=\abs{u_i c_5v_i c_5}$, which then fixes $x_l$ to be exactly $u_i c_5v_i c_5$. If there are only two occurrences of $c_5$ in $x_l^{r_l}$, then $n=1$, $x_l = u'_i c_5 v_i c_5 u''_i$ and $x_l^{r_l} = u'_i c_5 v_i c_5 u''_i u'''_i$. Then $r_l=\fe(u'_i c_5 v_i c_5 u''_i u'''_i) \leq 1 + \frac{\abs{u'''_i}}{\abs{u'_i c_5 v_i c_5 u''_i}} \leq 1 + \frac{\abs{u'_i}}{\abs{u'_i c_5 v_i c_5}} \leq 1 + \frac{\abs{u_i}}{\abs{u_i c_5 v_i c_5}} = 1 + \frac{1}{k+1}$.

    \bigskip

    So we see that the factor sequence leading to $\ace$ of $w$ is the obvious one, leading to $\ace(w)=n+\frac{1}{k+1}$. As previously stated, this leads to $\ace(h(w))=n+\frac{1}{k+1} \leq n + \lambda$. The second claim is easier to prove.

    \bigskip

CLAIM 2 ($\acei(h(w)) \geq  n+1-\delta$):

    Let $\psi_l (a)=a$ and $\psi_l (b)=b^l$. Then by the choice of $m$:
    \begin{align*}
    \hspace{0.5cm} \acei(h(w)) &\geq \lim_{l \to \infty} \lim_{i \to \infty} \fe(\psi_l (h((u_i c_5 v_i c_5 )^n u_i c_5))) \\
    &\geq n+ \lim_{l \to \infty} \lim_{i \to \infty} \frac{\abs{\psi_l (h(u_i c_5))}}{\abs{\psi_l (h(u_i c_5 v_i c_5))}}  \\
    &\geq n + \lim_{l \to \infty} \lim_{i \to \infty} \frac{\abs{h(u_i)}_b l}{\abs{h(u_i)}_b l + \abs{h(v_i)}_b l + 2\abs{h(c_5)}_b l +\abs{h(u_i)}_a + \abs{h(v_i)}_a} \\
    &\geq n + \lim_{l \to \infty} \lim_{i \to \infty} \frac{(m-2)\abs{u_i}l}{m\abs{u_i}l + 2\abs{v_i}l +6l} \\
    &=n + \lim_{l \to \infty} \lim_{i \to \infty} \frac{(m-2)\abs{u_i}l}{m\abs{u_i}l + 2(k(\abs{u_i}+1)-1)l +6l}  \\
    &\geq n + \lim_{l \to \infty} \frac{(m-2)l}{ml + 2kl}\\
    &= n + \frac{(m-2)}{m + 2k} \\
    &\geq n+1-\delta.
    \end{align*}
\end{proof}

\section{Conclusions and open problems}

We investigated how much we can introduce repetitions to words via injective morphisms.

For finite binary words, the fourth statement in Theorem $\ref{thm:fei-infty}$ reduces to a much simpler statement, because the suffix- and prefix-comparability conditions are automatically satisfied. Thus we think it is a good characterization of binary words $w$ with $\fei(w)=\infty$. In the case of larger alphabets the situation is more complicated, which leads to the following question:

\begin{open_problem}
    Find a simpler form for the last condition of Theorem $\ref{thm:fei-infty}$,
    or estimate the complexity of checking whether it holds for a given word.
\end{open_problem}

Similarly, for infinite words we have powerful theorems concerning binary words
that fail in the general case.
In the binary case, by Theorem $\ref{thm: uniform-binary}$,
an injective morphism cannot increase the asymptotic critical exponent by one or more.
When the alphabet size is $2n$, then, by Theorem $\ref{thm:big-acei}$,
the asymptotic critical exponent can sometimes increase at least by a factor of $n$.
Theorem $\ref{thm:general-upperbound}$ shows
that an increase by an arbitrarily large factor is not possible if the alphabet is fixed,
but the upper bound here is most likely not very good.
This makes the following question interesting:

\begin{open_problem}
    Let $\Sigma$ be an alphabet with $\abs{\Sigma}=n$.
    What is the value of
    \begin{displaymath}
        \sup_{\substack{w \in \Sigma^\N \\ \ace(w)\in \R} } \frac{\acei(w)}{\ace(w)}?
    \end{displaymath}
\end{open_problem}

We conjecture that the answer is linear with respect to $n$.
\bibliography{ref}

\end{document}